\documentclass[12pt]{amsart}
\usepackage{amssymb}
\usepackage{verbatim}
\usepackage[usenames]{color}
\usepackage{hyperref}
\usepackage{url}

\newtheorem{theorem}{Theorem}[section]
\newtheorem{proposition}[theorem]{Proposition} 
 
\newtheorem{lemma}[theorem]{Lemma}
\newtheorem{cor}[theorem]{Corollary}
\newtheorem{fact}[theorem]{Fact}

\theoremstyle{definition}
\newtheorem{definition}[theorem]{Definition}

\newtheorem{question}[theorem]{Question}

\theoremstyle{remark}
 
\newtheorem{remark}[theorem]{Remark}

\newcommand{\zfc}{\textnormal{ZFC}}
\newcommand{\zf}{\textnormal{ZF}}
\newcommand{\dc}{\textnormal{DC}}

\newcommand{\ad}{\textnormal{AD}}

\newcommand{\cf}{\mbox{cf}}
\newcommand{\mf}{\mathfrak}
\newcommand{\mc}{\mathcal}
\newcommand{\mbb}{\mathbb}

\newcommand{\baire}{{{^\omega}\omega}}

\title[Bounding 2D Functions by Products of 1D Functions]
 {Bounding 2D Functions \\ by Products of 1D Functions}
\author{Fran\c{c}ois Dorais}
\author{Dan Hathaway}
\date{\today}

\address{
Department of Mathematics \\
University of Vermont\\
Innovation Hall \\
82 University Place\\
Burlington, VT 05405 U.S.A.}
\email{Francois.Dorais@uvm.edu}

\address{
Department of Mathematics \\
University of Vermont\\
Innovation Hall \\
82 University Place\\
Burlington, VT 05405 U.S.A.}
\email{Daniel.Hathaway@uvm.edu}
\urladdr{\url{https://www.uvm.edu/cems/mathstat/profiles/daniel-hathaway}}

\begin{document}

\begin{abstract}
Given sets $X,Y$ and a regular cardinal $\mu$,
 let $\Phi(X,Y,\mu)$ be the statement
 that for any function $f : X \times Y \to \mu$,
 there are functions
 $g_1 : X \to \mu$ and
 $g_2 : Y \to \mu$
 such that
 for all $(x,y) \in X \times Y$,
 $$f(x,y) \le \max \{ g_1(x), g_2(y) \}.$$
In $\zfc$, the statement $\Phi(\omega_1, \omega_1, \omega)$
 is false.
However, we show the theory
 $\zf$ + ``the club filter on $\omega_1$ is normal''
 + $\Phi(\omega_1, \omega_1, \omega)$
 (which is implied by
 $\zf + \ad$)
 implies that for every $\alpha < \omega_1$
 there is a $\kappa \in (\alpha,\omega_1)$
 such that in some inner model,
 $\kappa$ is measurable
 with Mitchell order $\ge \alpha$.

There was an error in Welch's paper ``Characterizing Subsets of $\omega_1$ Constructible From a Real'',
 which he has retracted in a personal communication.
Our paper originally referenced that paper.
In this version of our paper, we are not using that result.
Our consistency strength upper bound has changed accordingly.
\end{abstract}

\maketitle

\section{Introduction}

\begin{remark}
In the published version of this paper
 we claimed that $\Phi(\omega_1,\omega_1,\omega)$
 was implied by $\zf + \dc + $``$V = L(\mathbb{R})$''$ +$``$\omega_1$ is measurable.''
The error came from
 Corollary 6 of \cite{Welch},
 which claimed that $\zf + \dc + $``$\omega_1$ is measurable''
 implies that every subset of $\omega_1$ is constructible from a real.
The author of \cite{Welch} agreed with us (in a personal communication)
 that the Corollary 6 of that paper is incorrect.
So, the revised claim we are making is that
 $\zf + \ad$ implies $\Phi(\omega_1,\omega_1,\omega)$.
\end{remark}

The Axiom of Choice allows us to define
 objects of size $\omega_1$ that
 do not behave like objects of size $\omega$.
For infinite cardinals $\lambda_1$ and $\lambda_2$,
 we introduce a statement
 $\Phi(\lambda_1, \lambda_2, \omega)$
 which implies in some sense
 that functions from
 $\lambda_1 \times \lambda_2$ to $\omega$
 behave like functions from
 $\omega \times \omega$ to $\omega$.
In Section 2 we show in $\zfc$ that
 $\Phi(\lambda, \omega, \omega)$ is true
 iff $\lambda < \mf{b}$,
 where $\mf{b}$ is the \textit{bounding number}.
We also show in $\zf$ that
 $\Phi(\mathbb{R},\omega,\omega)$ is false.
In Section 3 we show that
 $\Phi(\omega_1, \omega_1, \omega)$ is false in $\zf$
 assuming $\omega_1$ is regular and there
 is an injection of $\omega_1$ into $\mathbb{R}$.
Thus, if $\zf$ + ``$\omega_1$ is regular'' +
 $\Phi(\omega_1, \omega_1, \omega)$ holds,
 then $\omega_1$ is inaccessible in every inner model
 that satisfies the Axiom of Choice.

On the other hand consider $\ad$,
 the Axiom of Determinacy.
This is an axiom which contradicts
 the Axiom of Choice.
It imposes regularity properties on
 both subsets of $\mbb{R}$ and on
 bounded subsets of $\Theta$,
 which is the smallest ordinal that
 $\mbb{R}$ cannot be surjected onto.
$\ad$ implies in some sense that subsets
 of $\omega_1$ behave like subsets of $\omega$.
Contrasting with the Axiom of Choice,
 in Section 4 we show that $\ad$ implies
 $\Phi(\omega_1, \omega_1, \omega)$.

In Section 5, we aim to show that
 $\zf + \Phi(\omega_1, \omega_1, \omega)$
 implies that there are inner models with large cardinals.
We show that
 $\zf$ + ``the club filter on $\omega_1$ is normal''
 + $\Phi(\omega_1, \omega_1, \omega)$ implies
 1) every real has a sharp and 2)
 there is an inner model with a measurable cardinal.
We then go on to show that there are inner models
 with measurables of modestly high Mitchell order.
However we do not know how to produce an inner model
 with a measurable $\kappa$ of Mitchell order $\ge \kappa^+$,
 for example.

Here are the implications
 between various statements related to $\Phi$:
\begin{cor}
Each item below implies the next:
\begin{itemize}
\item[1)] $\ad$,
\item[2)] $\Phi(\omega_1, \omega_1, \omega)$
 + ``the club filter on $\omega_1$ is normal'',
\item[3)]
 either there is an inner model with a Woodin cardinal,
 or there is not and
 for some inner model $J$ which contains $K$,
 we have
 for each $\alpha < \omega_1$
 there exists $\kappa \in (\alpha,\omega_1)$
 such that $\kappa$ is measurable of Mitchell order $\ge \alpha$
 in $K^J$.
\end{itemize}
\end{cor}

We will use the following concepts
 several times:
\begin{definition}
An \textit{inner model}
 is a transitive model of $\zf$
 which contains all the ordinals.
$L$ is the smallest inner model.
\end{definition}

\begin{definition}
Given a set $X$, $L(X)$ represents the smallest
 inner model $M$ such that $X \in M$.
On the other hand, $L[X]$ represents the smallest
 inner model $M$ such that $X \cap M \in M$.
\end{definition}

For any set $X$, the inner model $L[X]$ exists
 (and certainly $L(X)$ exists as well).
For a set of ordinals $A$, we have
 $A \in L[A]$,
 and so $L[A] = L(A)$.
That is, if $A$ is a set of ordinals,
 then $L[A]$ is the smallest inner model $M$
 such that $A \in M$.
For a set of ordinals $A$,
 we will use the $L[A]$ notation instead
 of the $L(A)$ one.

\begin{definition}
Given any inner model $M$
 and a set of ordinals $A$,
 $M[A]$ is the smallest inner model $N$
 such that $M \subseteq N$ and $A \in N$.
\end{definition}
If $A$ is a set of ordinals,
 then $L[A]$ satisfies the Axiom of Choice.
We will need the following
 strengthening
 (see \cite{Grigorieff} page 451 for a proof):
\begin{proposition}
\label{zfc_preservation}
Let $M$ be an inner model.
Let $A$ be a set of ordinals.
Then $M[A]$ exists.
If moreover $M$ satisfies the Axiom of Choice,
 then so does $M[A]$.
\end{proposition}

\subsection{Consistency strength of related theories}

First, consider the theory
 $$T_1 := \zf + ``\mbox{the club filter on $\omega_1$ is an ultrafilter''.}$$
It is apparently unknown whether $T_1$
 is equiconsistent with $\zfc$.
Now when we add $\dc$ to $T_1$,
 we get that the club filter on $\omega_1$ is countably complete,
 and hence $\omega_1$ is a measurable cardinal.
The assumption (in $\zf$) that $\omega_1$ is a measurable cardinal
 implies there is an inner model of $\zfc$ with a measurable cardinal,
 by the $L[\mu]$ construction.
Jech showed that adding $\dc$ to the theory
 $\zf +$ ``$\omega_1$ is a measurable cardinal'' does not increase its consistency strength.
That is, Jech showed that
 $$T_2 := \zf + \dc + ``\omega_1 \mbox{ is a measurable cardinal''}$$
 is equiconsistent with $\zfc$ + ``there exists a measurable cardinal''.

We can strengthen $T_1$ and $T_2$ simultaneously to get
 $$T_3 := \zf + \dc + \mbox{``the club filter on $\omega_1$ is an ultrafilter''}.$$
Note that in this theory, the club filter on $\omega_1$ is a
 (countably complete) normal ultrafilter,
 witnessing that $\omega_1$ is measurable.
Mitchell \cite{mitchell_applications_of_covering}
 showed that $T_3$ implies there is an inner model of $\zfc$
 with a cardinal with a weak repeat point.

The consistency of the existence of 
 a cardinal $\kappa$ of Mitchell order $\kappa^{++}$
 strictly implies the consistency of the existence of a cardinal
 with a weak repeat point,
 which in turn strictly implies the consistency of the existence
 of a cardinal $\kappa$ of Mitchell order $\kappa^{+}$,
 which strictly implies the consistency of
 the existence of an inner model $M$
 such that $(\forall \alpha < \omega_1)(\exists \kappa \in (\alpha,\omega_1))\,
 \kappa$ has Mitchell order $\ge \alpha$ in $M$.

Finding an upper bound for $T_3$ is more problematic.
Radin and Woodin
 showed that if the existence of a
 $\mathcal{P}^2$-measurable cardinal is
 consistent, then so is $T_3$
 (see \cite{mitchell_how_weak} for a discussion).
Mitchell improved this in \cite{mitchell_how_weak}
 to get that if
 it is consistent with $\zfc$
 that there is a measurable cardinal
 $\kappa$ of Mitchell order $\kappa^{++}$,
 then $T_3$ is consistent.
Finally, Mitchell has worked towards reducing
 the consistency strength upper bound for $T_3$ down to
 $\zfc$ plus ``there exists a cardinal with
 a weak repeat point''.

Next, the club filter on $\omega_1$ being an ultrafilter
 is implied by 1) every real having a sharp together with
 2) every subset of $\omega_1$ being constructible from a real.
Finally, 1) and 2) follow from $\ad$.
We discuss this in Section~\ref{sec_upper_bound}.



\subsection{Acknowledgments}

We would like to thank
 Cody Dance and Trevor Wilson
 for discussions relating to this paper,
 especially concerning the behavior of
 sets of ordinals
 assuming $\ad$.
We would like to thank Phil Tosteson
 who originally asked about the status of
 $\Phi(\lambda_1, \lambda_2, \omega)$.
We would like to thank Andreas Blass
 for discussing the Dodd-Jensen core model.
Finally, we would like to thank
 Arthur Apter
 and Asaf Karagila
 concerning the consistency strength
 of $\omega_1$ being measurable
 versus the consistency strength
 of the club filter on $\omega_1$
 being a normal ultrafilter.

\section{Definition and the Bounding Number}

\begin{definition}
For sets $X, Y$
 and an infinite regular cardinal $\mu$,
 let $\Phi(X,Y,\mu)$ be the statement
 that whenever $f : X \times Y \to \mu$,
 there are functions
 $g_1 : X \to \mu$ and
 $g_2 : Y \to \mu$
 such that
 for all $(x,y) \in X \times Y$,
 $$f(x,y) \le \max \{ g_1(x), g_2(y) \}.$$
\end{definition}
We will mostly consider the case that
 $X$ and $Y$ are cardinals (which are ordinals).
When this happens,
 without loss of generality
 $\lambda_1, \lambda_2 \ge \mu$.
In this section we will consider arbitrary $\mu$,
 but for the rest of the paper we will focus
 on the $\mu = \omega$ case.

\begin{theorem}[$\zf$]
Let $\mu$ be a regular cardinal.
Let $X$ be a set.
Then $\Phi(X,\mu,\mu)$ holds iff
 each family $\langle f_x : x \in X \rangle$
 of functions from $\mu$ to $\mu$
 is eventually dominated (mod $<\mu$)
 by some single function.
\end{theorem}
\begin{proof}
First assume every such family indexed by $X$
 can be eventually dominated by a single
 function from $\mu$ to $\mu$.
Fix $f : X \times \mu \to \mu$.
Let $g_2 : \mu \to \mu$ be such that
 for each $x \in X$, the set
 $S_x := \{ y < \mu : f(x,y) > g_2(y) \}$
 has size $< \mu$.
Let $g_1 : X \to \mu$
 be such that for each $x \in X$,
 $g_1(x) \ge \sup \{ f(x,y) : y \in S_x \}$.
The functions $g_1$ and $g_2$ work.

For the other direction,
 assume there is a family indexed by $X$
 of functions from $\mu$ to $\mu$
 that cannot be eventually dominated
 by a single function.
Using such a family,
 let $f : X \times \mu \to \mu$
 be such that the functions of the form
 $y \mapsto f(x,y)$ cannot be
 eventually dominated by a single
 function from $\mu$ to $\mu$.
By possibly increasing the values
 of the function $f$,
 we may assume that whenever
 $x \in X$ and
 $Y \in [\mu]^\mu$,
 the set $f`` \{x\} \times Y$
 is unbounded below $\mu$.
Now, consider any $g_2 : \mu \to \mu$.
By hypothesis, fix some
 $x \in X$ such that
 $g_2$ does not eventually dominate
 the function $y \mapsto f(x,y)$.
That is,
 $S_x := \{ y \in \mu :
 f(x,y) > g_2(y) \}$ is in $[\mu]^\mu$,
 so $f`` \{x\} \times S_x$
 is unbounded below $\mu$.
Hence, there is no $\alpha < \mu$ such that
 for all $y < \mu$,
 $f(x,y) \le \max \{ \alpha, g_2(y) \},$
 so there can be no function $g_1$
 which works (because no value of
 $\alpha = g_1(x)$ works).
\end{proof}

Every family of size $\mu$ of functions from
 $\mu$ to $\mu$ is eventually dominated by a single
 function.
On the other hand, there exists a family indexed by ${^\mu 2}$
 of functions from $\mu$ to $\mu$ that cannot be
 eventually dominated by any single function.
This gives us the following:
\begin{cor}[$\zf$]
\label{zf_no_R_ord}
Let $\mu$ be a regular cardinal.
Then $\Phi(\mu, \mu, \mu)$ holds but
 $\Phi({^\mu 2}, \mu, \mu)$ does not.
\end{cor}

\begin{cor}[$\zf$]
\label{no_injection_observation}
Let $\mu$ be a regular cardinal.
Assume there is no injection of $\mu^+$ into ${^\mu 2}$.
Then for any cardinal $\lambda$ (that is an ordinal),
 $\Phi(\lambda,\mu,\mu)$ holds.
\end{cor}
\begin{proof}
Any size $\lambda$ family of functions from
 $\mu$ to $\mu$ has at most $\mu$ distinct members,
 and hence the family can be eventually dominated
 by a single function.
\end{proof}

Recall that in $\zfc$,
 $\mf{b}(\mu)$ is the cardinality of the smallest
 family of functions from $\mu$ to $\mu$
 that cannot be eventually dominated by a single function.
See \cite{Blass} for more on $\mathfrak{b}(\omega)$.
\begin{cor}[$\zfc$]
Let $\mu$ be a regular cardinal.
Let $\lambda$ be a cardinal.
Then $\Phi(\lambda,\mu,\mu)$ holds
 iff $\lambda < \mf{b}(\mu)$.
\end{cor}

\section{First Countable Topologies}

By Corollary~\ref{zf_no_R_ord},
 $\Phi(\mathbb{R},\omega,\omega)$ fails in $\zf$,
 and hence so does
 $\Phi(\mathbb{R}, \omega_1, \omega)$ and
 $\Phi(\mathbb{R}, \mathbb{R}, \omega)$.
In this section, we will show
 in $\zf$
 that if $\omega_1$ is regular
 and there is an injection of $\omega_1$
 into $\mathbb{R}$,
 then $\Phi(\omega_1, \omega_1, \omega)$
 fails as well.

Given a topological space $X$
 and a point $x \in X$,
 we call $\mc{B} \subseteq \mc{P}(X)$
 a neighborhood basis
 for $x$ iff
 1) every $U \in \mc{B}$ is open and contains $x$
 and 2) for every open $W \ni x$,
 there is some $U \in \mc{B}$ such that
 $U \subseteq W$.
We call $X$ \textit{first-countable}
 iff every $x \in X$ has a countable
 neighborhood basis.

\begin{definition}
Let $X$ be a topological space.
We call $X$ \textit{uniformly first-countable}
 iff there is a function $F$ with domain $X$
 such that for each $x \in X$,
 $F(x)$ is a sequence $\langle U_{x,n} : n < \omega \rangle$
 such that $\{ U_{x,n} : n < \omega \}$
 is a neighborhood basis for $x$.
\end{definition}

Assuming the Axiom of Choice,
 a topological space is first-countable
 iff it is uniformly first-countable.
Recall that a topological space $X$ is $T_1$
 iff for any distinct $x,y \in X$,
 there is an open set which contains $x$ but not $y$.

\begin{theorem}[$\zf$]
\label{metric_space_thm}
Let $X$ be a $T_1$ and
 uniformly first-countable
 topological space.
Let $Y \subseteq X$
 and assume $\Phi(X, Y, \omega)$.
Then we may write
 $Y = \bigcup_{n < \omega} Y_n$
 where each $Y_n$ is a closed discrete subset of $X$
 and $Y_0 \subseteq Y_1 \subseteq ...$.
\end{theorem}
\begin{proof}
Let the function $F$ witness that $Y$
 is uniformly first-countable, where
 $F(x) = \langle U_{x,n} : n < \omega \rangle$.
Assume without loss of generality that for each $x \in X$,
 we have $U_{x,0} \supseteq U_{x,1} \supseteq ...$.
Let $f : X \times Y \to \omega$ be defined by
 $$f(x,y) := \begin{cases}
 \min \{ n < \omega : y \not\in U_{x,n} \} & \mbox{if } x \not= y, \\
 0 & \mbox{if } x = y.
 \end{cases}$$
Thus for all distinct $x \in X$ and $y \in Y$,
 $$y \not\in U_{x,f(x,y)}.$$
By $\Phi(X,Y,\omega)$,
 there are functions $g_1 : X \to \omega$ and  $g_2 : Y \to \omega$
 such that $f(x,y) \le \max \{ g_1(x), g_2(y) \}$
 for all distinct $x \in X$ and $y \in Y$.
Thus $$y \not\in U_{x,\max\{g_1(x),g_2(y)\}}$$
 for all distinct $x \in X$ and $y \in Y$.
Define $Y_n := \{ y \in Y : g_2(y) \le n \}$.
Then
 $$y \not\in U_{x,\max\{g_1(x),n\}} =: N(x,n)$$
 for any distinct $x \in X$ and $y \in Y_n$.
Since $N(x,n)$ is independent of $y \in Y_n$,
 we have that $N(x,n)$ is a neighborhood of $x$
 whose only possible point in $Y_n$ that it contains
 is $x$ itself.
That is,
 $$N(x,n) \cap Y_n =
 \begin{cases}
 \emptyset & \mbox{if } x \not\in Y_n, \\
 \{x\} & \mbox{if } x \in Y_n.
 \end{cases}$$\
It follows that $Y_n$ is a discrete space.
Also note that since each point $x$ in $X - Y_n$
 has a neighborhood $N(x,n)$ disjoint from $Y_n$,
 we have $X - Y_n$ is open,
 and hence $Y_n$ is closed.
Also note that $Y = \bigcup_{n < \omega} Y_n$
 and $Y_0 \subseteq Y_1 \subseteq ...$.
\end{proof}

\begin{cor}[$\zf$]
\label{no_injection_cor}
Assume $\omega_1$ is regular and
 there is an injection of $\omega_1$ into $\mathbb{R}$.
Then $\Phi(\omega_1, \omega_1, \omega)$ is false.
\end{cor}
\begin{proof}
Assume, towards a contradiction,
 that $\Phi(\omega_1, \omega_1, \omega)$ holds.
By hypothesis, let $f : \omega_1 \to \baire$
 be an injection.
Let $Y$ be the range of $f$.
We have that $\Phi(Y,Y,\omega)$ holds.
Since $\baire$
 is $T_1$ and uniformly first-countable,
 so is $Y$.
By Theorem~\ref{metric_space_thm},
 $Y$ is the union of a countable
 chain of discrete spaces.
Since $\baire$ is second countable,
 so is $Y$,
 and hence every discrete subspace of $Y$
 is countable.
Hence $Y$ is the union of countably
 many countable sets.
Using the bijection $f : \omega_1 \to Y$,
 we see that $\omega_1$
 is a countable union of countable sets.
This contradicts $\omega_1$
 being regular.
\end{proof}

\begin{cor}[$\zfc$]
$\Phi(\omega_1, \omega_1, \omega)$ is false.
\end{cor}

We will use the following later:
\begin{lemma}[$\zfc$]
\label{some_ords_are_semi_met}
Let $S$ be a set of ordinals such that
 each $\alpha \in S$ is either a successor ordinal
 or has countable cofinality.
Then $S$, with its usual order topology,
 is $T_1$ and uniformly first-countable.
\end{lemma}
\begin{proof}
We will define the appropriate function
 $F(\beta) = \langle U_{\beta,n} : n < \omega \rangle$.
For each successor ordinal $\beta \in S$,
 define $F(\beta)$ to be the sequence
 $\langle \{\beta\}, \{\beta\}, \{\beta\}, ... \rangle$.
For each limit ordinal $\beta \in S$,
 use the Axiom of Choice to choose an $\omega$-sequence
 $\beta_0 < \beta_1 < ...$ with limit $\beta$.
Let $F(\beta)$ be the sequence
 $\langle (\beta_0,\beta], (\beta_1,\beta], ... \rangle$.
One can check that this works.
\end{proof}

\section{Consistency Strength Upper Bound}
\label{sec_upper_bound}

The main result of this section is that
 $\zf + \ad$
 implies $\Phi(\omega_1, \omega_1, \omega)$.
The following is useful:
\begin{proposition}[$\zf + \dc$]
\label{dc_and_club_ultra_implies_normal}
The following are equivalent:
\begin{itemize}
\item[1)] The club filter on $\omega_1$
 is an ultrafilter.
\item[2)] The club filter on $\omega_1$
 is a (countably complete) normal ultrafilter.
\end{itemize}
\end{proposition}
\begin{proof}
Certainly 2) implies 1),
 so all we have to do is show 1) implies 2).
Assume the club filter on $\omega_1$
 is an ultrafilter.
First note that if the club
 filter on $\omega_1$ is an ultrafilter,
 then because we are assuming $\dc$ we have
 that the club filter on $\omega_1$ is countably complete.
In fact the Countable Axiom of Choice is sufficient for this.
This is because given an $\omega$-sequence of sets
 in the club filter,
 we can form a sequence of clubs contained in each one.
And, the countable intersection of clubs is club.

We will now show that club filter is normal.
Let $\langle A_\beta : \beta < \omega_1 \rangle$
 be an $\omega_1$-sequence of sets in
 the club filter on $\omega_1$.
Let $$Y := \{ \alpha < \omega_1 :
 (\forall \beta < \alpha)\, \alpha \in A_\beta \}$$
 be their diagonal intersection.
We must show that $Y$ contains a club.
Because we are assuming
 the club filter is an ultrafilter,
 it suffices to show that $Y$ is stationary,
 because then either $Y$ contains a club or $\omega_1 - Y$
 does, but $\omega_1 - Y$ cannot contain a club because
 $Y$ is stationary.

Let $Z \subseteq \omega_1$ be a club.
We will show that $Z$ intersects $Y$.
We will construct an $\omega$-sequence
 $\langle C_n : n < \omega \rangle$
 of club subsets of $\omega_1$
 using DC.
\begin{itemize}
\item Pick $C_0 \subseteq Z \cap A_0$
 such that $0 \not\in C_0$.
\item Having chosen $C_n$,
 let $\gamma_n = \min C_n$
 and pick $C_{n+1} \subseteq
 C_n \cap Z \cap \bigcap_{\alpha < \gamma_n} A_\alpha$
 such that $\gamma_n \not\in C_{n+1}$.
\end{itemize}
Let $\gamma := \sup \gamma_n$.
Note that $(\forall n \in \omega)\, \gamma \in C_n$.
We see that $\gamma$ belongs to $Y \cap Z$.
Thus $Y \cap Z$ is non-empty,
 which is what we wanted to show.
\end{proof}

Here is the relationship
 between some relevant statements:

\begin{proposition}[$\zf + \dc$]
\label{equiv_of_measurable}
Each item implies the one below it
 (and $2$ and $3$ are equivalent):
\begin{itemize}
\item[1)]
\begin{itemize}
\item[1a)] $(\forall r \in \mathbb{R})\, r^\#$ exists,
\item[1b)] $(\forall A \subseteq \omega_1)
 (\exists r \in \mathbb{R})\, A \in L[r]$.
\end{itemize}
\item[2)] The club filter on $\omega_1$ is
 an ultrafilter.
\item[3)] The club filter on $\omega_1$ is a
 (countably complete) normal ultrafilter.
\item[4)] $\omega_1$ is a measurable cardinal.
\end{itemize}
\end{proposition}
\begin{proof}

1) $\Rightarrow$ 2):
 we will show that the club filter is ultra.
Let $A \subseteq \omega_1$.
Let $r \in \mathbb{R}$ be such that $A \in L[r]$.
Using that $r^\#$ exists,
 for a final segment of $L[r]$-indiscernibles $\alpha < \omega_1$,
 either $\alpha \in A$ or $\alpha \not\in A$.
A final segment of $L[r]$-indiscernibles is a club,
 so this shows that the club filter is ultra.

2) $\Rightarrow$ 3):
 This is by
 Proposition~\ref{dc_and_club_ultra_implies_normal}.

3) $\Rightarrow$ 4):
 If there is a non-principal
 normal (or even just a countably complete)
 ultrafilter on $\omega_1$,
 then $\omega_1$ is measurable by definition.
\end{proof}

Here is the proof of the consistency of
 $\Phi(\omega_1, \omega_1, \omega)$:

\begin{theorem}[$\zf$]
\label{main_thm}
Assume
\begin{itemize}
\item $(\forall r \in \mathbb{R})\, r^\#$ exists,
\item $(\forall A \subseteq \omega_1)(\exists r \in \mathbb{R})\,
 A \in L[r]$.
\end{itemize}
Then $\Phi(\omega_1, \omega_1, \omega)$ holds.
\end{theorem}
\begin{proof}
Let $f : \omega_1 \times \omega_1 \to \omega$.
By hypothesis, fix $r \in \mathbb{R}$
 such that $f \in L[r]$.
Let $\mc{I} \subseteq \mbox{Ord}$ be the club of
 Silver indiscernibles for $L[r]$.
Let $a_1, ..., a_n \in \mc{I}$ be the indiscernibles
 used in the definition of $f$ in $L[r]$.
That is, there is a Skolem term $t$ such that
 $t(a_1, ..., a_n)^{L[r]} = f$.
In $L[r^\#]$,
 for each $\alpha < \omega_1$ choose
 $C_\alpha := (b^1_\alpha, ..., b^{m_\alpha}_\alpha, t_\alpha)$
 to be some finite
 sequence of $L[r]$-indiscernibles
  $b^1_\alpha, ..., b^{m_\alpha}_\alpha \in \mc{I}$
  and a Skolem term $t_\alpha$ such that
 $$\alpha =
 t_\alpha(b^1_\alpha, ..., b^{m_\alpha}_\alpha)^{L[r]}.$$
The function $\alpha \mapsto C_\alpha$ is in $L[r^\#]$.
Let $\mbox{Type}(\alpha) := (m_\alpha, t_\alpha)$,
 and let $\mc{T}$ be the collection of all possible types.
Note that $\mc{T}$ is countable.
The value of $f(\alpha_1, \alpha_2)$,
 since it is a natural number, depends only on
 the Skolem terms $t$, $t_{\alpha_1}$, $t_{\alpha_2}$ and on
 the relative ordering of the indiscernibles in
 $C_{\alpha_1}$ and $C_{\alpha_2}$ with each other
 and with $a_1, ..., a_n$.
Knowing just $\mbox{Type}(\alpha_1)$
 and $\mbox{Type}(\alpha_2)$,
 only finitely many such orderings
 of the indiscernibles are possible,
 and so only finitely many values of $f(\alpha_1, \alpha_2)$
 are possible.
Given any $T_1, T_2 \in \mc{T}$,
 define $f''(T_1, T_2)$ to be the
 maximum of these possible values.
We have $f'' : \mc{T} \times \mc{T} \to \omega$.

Define $f' : \omega_1 \times \omega_1 \to \omega$ by
 $f'(\alpha_1, \alpha_2) = f''(\mbox{Type}(\alpha_1),
 \mbox{Type}(\alpha_2))$.
We have that
 $f \le f'$ ($f'$ everywhere dominates $f$).
If we can find functions
 $g'_1, g'_2 : \omega_1 \to \omega$ such that
 $(\forall x, y \in \omega_1)
 \, f'(x,y) \le \max \{ g'_1(x), g'_2(y) \}$,
 then we will be done.
By the definitions of $f'$ and $f''$, it suffices to find
 $g''_1, g''_2 : \mc{T} \to \omega$ such that
 $(\forall T_1,T_2 \in \mc{T})\,
 f''(T_1,T_2) \le \max \{ g''_1(T_1), g''_2(T_2) \}$,
 because then we can define
 $g'_1(x) := g''_1(\mbox{Type}(x))$ and
 $g'_2(y) := g''_2(\mbox{Type}(y))$.
There must be such functions $g''_1$ and $g''_2$
 because $|\mc{T}| = \omega$.
\end{proof}


Solovay showed that
 $\ad$ implies $\omega_1$ is measurable
 (see Theorem 33.12 in \cite{Jech}).
Hence, $\ad$ implies every real has a sharp.
Also, Solovay proved under $\ad$ that
 every subset of $\omega_1$ is constructible
 from a real (see Lemma 2.8 of \cite{Kleinberg}).
Thus using the theorem above,
 we get the following:

\begin{cor}[$\zf$]
Assume $\ad$.
Then $\Phi(\omega_1, \omega_1, \omega)$ holds.
\end{cor}

But note also that $\ad$
 implies the club filter on $\omega_1$
 is an ultrafilter
 (see again Theorem 33.12 in \cite{Jech}).
So when we combine $\ad$ with $\dc$
 and use Proposition~\ref{dc_and_club_ultra_implies_normal}
 we get the following:
\begin{proposition}
$\ad + \dc$ implies the
 club filter on $\omega_1$
 is a (countably complete)
 normal ultrafilter.
\end{proposition}

Just as we observed in
 Corollary~\ref{no_injection_observation},
 note that if $\Phi(\omega_1, \omega_1, \omega)$
 holds and there is no injection of
 $\omega_2$ into $\mc{P}(\omega_1)$,
 then $\Phi(\lambda,\omega_1,\omega)$ holds
 for each ordinal $\lambda$.

It is natural to ask what fragments
 of the Axiom of Choice are consistent with
 $\Phi(\omega_1, \omega_1, \omega)$.
One such fragment consistent with
 $\Phi(\omega_1, \omega_1, \omega)$
 is the existence of a
 non-principal ultrafilter on $\omega$,
 which we will prove now.
Recall that $\omega \rightarrow (\omega)^\omega_2$
 is that statement that given any coloring
 of the infinite subsets of $\omega$ using two colors,
 there is an infinite subset of $\omega$
 all of whose infinite subsets have the same color.

\begin{theorem}
Assume $\ad$
 is consistent.
Then there is a model of $\zf$ in which
 $\Phi(\omega_1, \omega_1, \omega)$ holds
 but there is a non-principal ultrafilter on $\omega$.
\end{theorem}
\begin{proof}
Assume $V$ satisfies $\ad$.
Let $M_1 = L(\mathbb{R})$.
We have that $M_1 \models \ad$.
Thus $M_1$ satisfies $\Phi(\omega_1, \omega_1, \omega)$.
By a result of Mathias \cite{Mathias},
 since $M_1 \models ($``$V = L(\mathbb{R})$''$ + \ad)$,
 we have that $M_1$ satisfies
 $\omega \rightarrow (\omega)^\omega_2$.

Let $\mbb{P}$ be the $P(\omega)/\mbox{Fin}$ forcing
 of $M_1$.
Let $H$ be $\mbb{P}$-generic over $M_1$
 and let $M_2 := M_1[H]$.
Since $M_1 \models \omega \rightarrow (\omega)^\omega_2$,
 forcing with $\mbb{P}$
 \textit{will add no new sets of ordinals}
 (see \cite{Barren}).
Hence, $M_2 \models \Phi(\omega_1, \omega_1, \omega)$.
On the other hand,
 $M_2 \models$ ``there is a non-principal
 ultrafilter on $\omega$'',
 by the nature of the $\mathbb{P}$ forcing.
\end{proof}

Suppose we wanted to prove a result like
 Theorem~\ref{main_thm},
 but with $\Phi(\omega_2, \omega_2, \omega)$ in place
 of $\Phi(\omega_1, \omega_1, \omega)$.
Let $T_3$ be the tree of a scale of a complete
 $\Pi^1_3$ set of reals.
It follows from $\ad$ that every subset of
 $\omega_2$ (and even every subset of $\aleph_{\omega+1}$)
 is in $L[T_3,x]$ for some real $x$.
So, while we proved $\Phi(\omega_1, \omega_1, \omega)$
 using an analysis of the models $L[x]$ for $x \in \mathbb{R}$,
 it is reasonable to assume that a sufficient understanding
 of the $L[T_3,x]$ models will prove
 $\Phi(\omega_2, \omega_2, \omega)$.
Dance \cite{Dance} has made recent progress on
 an analysis of the $L[T_3,x]$ models,
 finding a system of indiscernibles for them.
However, we are unable to complete a proof
 that $\Phi(\omega_2, \omega_2, \omega)$
 holds assuming $\ad$.
Indeed, the consistency of $\Phi(\omega_2, \omega_2, \omega)$
 is still open.

Let us conclude this section
 with a way to reorganize
 Proposition~\ref{equiv_of_measurable}
 that may provide insight:
\begin{proposition}[$\zf + \dc$]
Assume that every subset of $\omega_1$
 is constructible from a real.
Then the following are equivalent:
\begin{itemize}
\item[1)] Every real has a sharp.
\item[2)] The club filter on $\omega_1$ is an ultrafilter.
\item[3)] The club filter on $\omega_1$
 is a (countably complete) normal ultrafilter.
\item[4)] $\omega_1$ is measurable.
\end{itemize}
\end{proposition}
\begin{proof}
1) $\Rightarrow$ 2): This is what we did in
 the 1) $\Rightarrow$ 2) direction
 of Proposition~\ref{equiv_of_measurable}.
Here we do not need $\dc$.

2) $\Rightarrow$ 3):
This follows from
 Proposition~\ref{dc_and_club_ultra_implies_normal}.
Here we only need $\dc$,
 not that every subset of $\omega_1$
 is constructible from a real.

3) $\Rightarrow$ 4): This is trivial.

4) $\Rightarrow$ 1): Let $\mc{U}$ be a normal
 measure on $\omega_1$.
Fix $r \in \mathbb{R}$.
We will show that $r^\#$ exists.
We have that $\omega_1$ is measurable in
 the model of $L[\mc{U},r]$,
 which satisfies $\zfc$.
So $L[\mc{U},r]$ thinks that $r^\#$ exists.
Hence, $r^\#$ actually exists.
\end{proof}

\section{Consistency Strength Lower Bound}

To show that $\Phi(\omega_1, \omega_1, \omega)$
 implies there are inner models with large cardinals,
 we will additionally assume some fragment
 of the Axiom of Choice.
The first such fragment
 we will consider
 is that $\omega_1$ is a regular cardinal:

\begin{proposition}[$\zf$]
\label{inacc_in_every_zfc}
Assume $\omega_1$ is regular and that
 $\Phi(\omega_1, \omega_1, \omega)$ holds.
Let $M$ be an inner model of $\zfc$.
Then $\omega_1$ is strongly inaccessible in $M$.
\end{proposition}
\begin{proof}
By Corollary~\ref{no_injection_cor},
 there is no injection of $\omega_1$ into $\mathbb{R}$.
By a well-known argument,
 this together with $\omega_1$ being regular
 implies that $\omega_1$
 is strongly inaccessible in $M$.
\end{proof}

For this section,
 we will use the following definition:

\begin{definition}
Let $M$ be an inner model.
Then $$S^M := \{ \alpha < \omega_1 : M \models \cf(\alpha) \le \omega \},$$
 $$R^M := \{ \alpha < \omega_1 : M \models
 \alpha \mbox{ is a regular cardinal} \}.$$
\end{definition}

\begin{lemma}($\zf$)
Assume $\omega_1$ is regular and
 that $\Phi(\omega_1, \omega_1, \omega)$ holds.
Let $M$ be an inner model of $\zfc$.
Then $S^M$
 is a non-stationary subset of $\omega_1$.
\end{lemma}
\begin{proof}
By Lemma~\ref{some_ords_are_semi_met},
 $S^M$ is $T_1$ and uniformly first-countable in $M$.
It follows from Theorem~\ref{metric_space_thm}
 that we can write
 $S^M = \bigcup_{n < \omega} S_n$
 where each $S_n$ is a closed and discrete
 subset of $S^M$,
 and where $S_0 \subseteq S_1 \subseteq ...$.
Since $S_n$ is a closed subset of $S^M$
 and $S^M$ has the subspace topology
 induced by $\omega_1$,
 we have $S_n = C_n \cap S^M$
 for some closed $C_n \subseteq \omega_1$.
Without loss of generality,
 let $C_n$ be the closure of $S_n$
 in $\omega_1$.
Let $C'_n$ be the set of limit points of $C_n$.
We claim that $C'_n \cap S^M = \emptyset$.

To see why, let $\alpha \in C'_n$.
Now $\alpha$ is a limit of points $\beta$ in $C_n$,
 where each such $\beta$ is either in $S_n$
 or is the limit of points in $S_n$.
Hence $\alpha$ is a limit of points in $S_n$.
Suppose towards a contradiction that $\alpha \in S^M$.
Fix $m \ge n$ such that $\alpha \in S_m$.
Then $\alpha \in S_m$ is the limit of a sequence
 of points in $S_m$,
 which contradicts $S_m$ being discrete.

Now note that
 $S^M$ is unbounded below $\omega_1$
 because of ordinals of the form $\gamma + \omega$.
Since $S^M = \bigcup_{n < \omega} S_n$
 and $\omega_1$ is regular in $V$,
 fix $n \in \omega$ such that $S_n$ is unbounded below $\omega_1$.
Then $C_n'$ is unbounded below $\omega_1$ and is also closed.
So $S^M$ is disjoint from the club $C_n'$,
 so $S^M$ is non-stationary.
\end{proof}

Now let us strengthen the hypothesis
 ``$\omega_1$ is regular'' to
 ``the club filter on $\omega_1$ is normal''.
Note this is indeed a strengthening because
 if $\omega_1$ was singular,
 then $\omega_1$ would be the union of
 countably many countable sets (which are
 therefore each non-stationary),
 so the non-stationary ideal on $\omega_1$
 would not be closed
 under countable unions.

Using this stronger fragment of the Axiom of Choice,
 we can show that $\Phi(\omega_1, \omega_1, \omega)$
 implies that every real has a sharp,
 that there is an inner model with a measurable cardinal,
 and even a bit more.
These proofs will factor through the following concept,
 which in this paper we will call
 the \textit{black box}.
\begin{definition}
\label{black_box}
The \textit{black box} is the statement that
 for any inner model $M$ of $\zfc$,
 the following hold:
\begin{itemize}
\item[1)] $\omega_1$ is strongly inaccessible in $M$,
\item[2)] $R^M$ contains a club subset of $\omega_1$.
\end{itemize}
\end{definition}

\begin{lemma}[$\zf$]
Assume the club filter on $\omega_1$ is normal
 and that $\Phi(\omega_1, \omega_1, \omega)$ holds.
Let $M$ be an inner model of $\zfc$.
Then $R^M$ contains a club subset of $\omega_1$.
\end{lemma}
\begin{proof}
For every $\beta < \omega_1$,
 $$X_\beta := \{ \alpha < \omega_1 : M \models \cf(\alpha) \le \beta \}$$
 is non-stationary,
 since $X_\beta \subseteq S^{M[s]}$ where $s \in \mathbb{R}$
 is such that $M[s] \models (\beta$ is countable$)$.
Let $C$ be the diagonal intersection
 of the sets $\omega_1 - X_\beta$
 for $\beta < \omega_1$.
Since each $\omega_1 - X_\beta$ contains a club
 and since the club filter on $\omega_1$ is normal,
 $C$ itself contains a club.
Now
 $$C = \{ \alpha < \omega_1 : (\forall \beta < \alpha)\, \alpha \not\in X_\beta \}.$$
That is,
 $$C = \{ \alpha < \omega_1 : (\forall \beta < \alpha)\,
 M \models \cf(\alpha) > \beta \}.$$
So $C = R^M$.
\end{proof}

\begin{cor}[$\zf$]
Assume the club filter on $\omega_1$ is normal
 and that $\Phi(\omega_1, \omega_1, \omega)$ holds.
Then the black box holds.
\end{cor}

\begin{definition}
Given an inner model $M$,
 we say $M$ \textit{computes singulars correctly}
 iff whenever $\kappa$ is a singular
 cardinal in $V$,
 then $\kappa$ is a singular cardinal in $M$.
\end{definition}

By a result of Jensen (see \cite{Mitchell}),
 for any real $r \in \mathbb{R}$,
 if $L[r]$ does not compute singulars correctly,
 then $r^\#$ exists.

\begin{proposition}[$\zf$]
Assume the black box holds.
Then for each $r \in \mathbb{R}$,
 $r^\#$ exists.
\end{proposition}
\begin{proof}
Fix $r \in \mathbb{R}$.
By the previous lemma,
 $R^{L[r]}$ contains a club $C \subseteq \omega_1$ (which is not in $L[r]$).
Now $L[r][C]$ is an inner model of $\zfc$,
 by Proposition~\ref{zfc_preservation}.
By 1) of the black box,
 $\omega_1$ is strongly inaccessible in $L[r][C]$.
So we may find a $\kappa \in C$
 that is a singular cardinal in $L[r][C]$.
Since $\kappa \in C$,
 $\kappa$ is regular in $L[r]$.
So $L[r]$ is not correct about the singulars
 of $L[r][C]$.
Hence $L[r][C] \models r^\#$ exists.
Any inner model that thinks it contains $r^\#$
 is right about this,
 so $V \models r^\#$ exists.
\end{proof}

For the next result, let $K$ refer to
 $K^{DJ}$, the Dodd-Jensen core model.
By a result of Dodd and Jensen (see \cite{Mitchell}),
 if there is no inner model with a measurable cardinal,
 then
 \begin{itemize}
 \item[1)] $K$ computes singulars correctly,
 \item[2)] if $N$ is any inner model such that $N \supseteq K$,
 then $K^N = K$.
 \end{itemize}

\begin{proposition}[$\zf$]
\label{inner_model_with_measurable}
Assume the black box holds.
Then there is an inner model with a measurable cardinal.
\end{proposition}
\begin{proof}
Suppose, towards a contradiction,
 that there is no inner model with a measurable cardinal.
Let $M$ be the Dodd-Jensen core model $K$ of $V$.
We have that $M$ is an inner model of $\zfc$.
By 2) of the black box,
 $R^{M}$ contains a club $C \subseteq \omega_1$\
 (which is not in $M$).
Now $M[C]$ is an inner model of $\zfc$
 by Proposition~\ref{zfc_preservation}.
By 1) of the black box,
 $\omega_1$ is strongly inaccessible in $M[C]$.
So we may find a $\kappa \in C$
 that is a singular cardinal in $M[C]$.
Now $\kappa$ is regular in $M$ but singular in $M[C]$.
Thus $M[C]$ thinks that $M$ does not
 compute singulars correctly.

But $M[C]$ has no inner model with a measurable cardinal
 (because $V$ has none).
So $M[C]$ thinks $K^{M[C]}$ computes singulars correctly.
But since $M[C]$ contains $M = K$,
 we have $K^{M[C]} = K = M$.
So $M[C]$ thinks $M$ both does and does not compute
 singulars correctly, which is impossible.
\end{proof}

\begin{remark}
Now that we have that the black box implies
 there is an inner model with a measurable cardinal,
 here is how we might show that $0^\dagger$ exists.
Assume the black box and fix
 an inner model $L[\mu]$
 with exactly one measurable cardinal $\kappa$.
It will follow from Proposition~\ref{heavy_duty}
 that we can fix such an $L[\mu]$ model
 such that $\kappa < \omega_1^V$,
 but for now just assume that this is the case.
Let $M = L[\mu]$.
Fix a club $C \subseteq \omega_1$
 subset of $R^M$.
Let $\lambda = \omega_1^V$.
The key is that $M[C] \models ($any forcing extension
 of $M$ by a poset $\mbb{P} \in M$
 of size $|\mbb{P}|^M < \lambda$
 will be wrong about the singulars of $M[C]$
 cofinally below $\lambda$).
In particular neither $M$ nor any Prikry generic
 forcing extension of $M$ (using the
 normal ultrafilter on $\kappa$)
 can be correct about the singulars of $M[C]$.
So by the covering theorem for $0^\dagger$
 (see \cite{Mitchell}),
 it must be that $M[C] \models (0^\dagger$ exists).
Hence $0^\dagger$ exists.
\end{remark}

A target somewhat beyond $0^\dagger$
 is an inner model $M$ with a measurable $\kappa$
 of high Mitchell order
 (but still with Mitchell order $< \kappa$).
Our strategy is to use the fact that for
 any singular cardinal $\kappa$
 of uncountable cofinality that is regular in
 ``the core model'',
 then the Mitchell order of $\kappa$ in
 the core model is at least $\cf^V(\kappa)$.
There are versions of this theorem with weaker
 and weaker anti-large cardinal hypothesis.
Mitchell proved a version,
 as did Sean Cox \cite{cox}.
We will use the version
 that assumes there is no inner model with
 a Woodin cardinal,
 due to Mitchell and Schimmerling:

\begin{fact}[Mitchell and Schimmerling \cite{schimmerling}]
\label{sean_cox_fact}
Assume there is no inner model with a Woodin cardinal.
Let $K$ be the core model.
Assume $\kappa > \omega_2$
 is a cardinal that
 is regular in $K$ but singular in $V$.
Assume that $\kappa$ has uncountable cofinality
 $\nu := \cf^V(\kappa)$ in $V$.
Then $K \models \kappa$ is measurable
 with Mitchell order $\ge \nu$.
\end{fact}
Note: The theorem in \cite{schimmerling}
 does not assume that $\kappa$ is a cardinal,
 but rather that $\kappa$ is an ordinal
 such that $\cf(\kappa) < |\kappa|$.

Our argument now will mirror
 Proposition~\ref{inner_model_with_measurable},
 where we start with $M =$ the core model,
 and we form $M[C]$ where $C \subseteq \omega_1$
 is a club of regular cardinals of $M$.
We then show that the core model 
 of $M[C]$ cannot compute singulars correctly.
In Proposition~\ref{inner_model_with_measurable}
 we had that the core model of $M[C]$
 was $M$ itself.
While we are not sure if we have that in
 the situation we are about to consider,
 we will have that the core model of $M[C]$
 is \textit{contained} in $M$.

We have this by using the Steel core model.
There is a $\Sigma_2$ formula $\psi_K(x)$ such that
 if there is no inner model
 with a Woodin cardinal, then
 $\{ x : \psi_K(x) \}$
 is an inner model
 known as the Steel core model $K^{Steel}$,
 which we simply denote as $K$.
\begin{fact} [Jensen and Steel \cite{k_without_measurable}]
Assume there is no inner model with a Woodin cardinal.
Let $\mu$ be an uncountable cardinal.
Then $K | \mu$ is $\Sigma_1$ definable over $L(H_\mu)$,
 uniformly in $\mu$.
\end{fact}
In particular, there is a $\Sigma_1$ formula $\varphi(x)$
 such that for any inner model $M$,
 for any uncountable cardinal $\mu$ of $M$
 and for any $x \in H_\mu^M$, we have
 $$(M \models \psi_K(x))
 \Leftrightarrow
 (L(H_\mu)^M \models \varphi(x)).$$

Note that if $X,Y$ are sets with $X \in Y$,
 then $L(X) \subseteq L(Y)$.
\begin{cor}
\label{core_model_shrinks}
Assume there is no inner model with a Woodin cardinal.
Let $N$ be an inner model.
Then $$K^N \subseteq K.$$
\end{cor}
\begin{proof}
It suffices to show for every uncountable cardinal $\mu$
 that $K^N | \mu \subseteq K$.
Fix such a $\mu$.
Fix $v \in K^N | \mu$.
So $N \models \psi_K(v)$.
Hence $L(H_\mu)^N \models \varphi(v)$,
 where $\varphi$ is the $\Sigma_1$ formula
 described above.
Now because $\varphi$ is $\Sigma_1$
 and $L(H_\mu)^N
 \subseteq L(H_\mu)$,
 we have $L(H_\mu) \models \varphi(v)$.
Thus
 $\psi_K(v)$ holds, and so
 $v \in K$.
\end{proof}

\begin{proposition}
\label{heavy_duty}
Assume the black box holds.
Then either there is an inner model
 with a Woodin cardinal,
 or else there is not and for some inner model $J$
 containing $K$,
 we have for each $\alpha < \omega_1$
 there exists $\kappa \in (\alpha, \omega_1)$
 such that
 $\kappa$ is measurable with Mitchell
 order $\ge \alpha$
 in $K^J$.
\end{proposition}
\begin{proof}
Assume there is no inner model
 with a Woodin cardinal.
Let $M := K = K^{Steel}$.
Since $M$ is an inner model of $\zfc$,
 by the black box fix a club $C \subseteq \omega_1^V$
 of regular cardinals of $M$.
Now fix $\alpha < \omega_1^V$.
Let $\kappa \in C$ be singular in $J = M[C]$
 with $\cf^{M[C]} \ge \alpha$.
Note that $\kappa$ is regular in $M$.

By Corollary~\ref{core_model_shrinks}
 we have $K^{M[C]} \subseteq M$.
Since $\kappa$ is regular in $M$,
 it is also regular in $K^{M[C]}$.
Now by Fact~\ref{sean_cox_fact}
 applied inside $M[C]$,
 we have that $\kappa$ is measurable in
 $K^{M[C]}$ with Mitchell order
 $\ge \alpha$.
\end{proof}

We do not see how to use \cite{mitchell_applications_of_covering},
 for example,
 to improve the lower bound of the black box to
 an inner model with a measurable $\kappa$
 of Mitchell order $\kappa$.

\section{Questions}

\subsection{Consistency strength related to $\Phi(\omega_1, \omega_1, \omega)$}

We know that $\zf$ + ``the club filter on $\omega_1$ is normal''
 + $\Phi(\omega_1, \omega_1, \omega)$ implies
 there are inner models with measurables of modestly
 high Mitchell order,
 but we do not know if there is an inner model
 with a measurable cardinal
 $\kappa$ of Mitchell order $\kappa$.
So the exact strength remains unknown:

\begin{question}
What is the consistency strength of
 $\zf$ + ``the club filter on $\omega_1$ is normal''
 + $\Phi(\omega_1, \omega_1, \omega)$?
\end{question}

So far our lower bounds factor through the \textit{black box}
 (Definition~\ref{black_box}),
 so we ask the following pure inner model theory question:
\begin{question}
What is the consistency strength
 of the black box?
\end{question}

Going in the other direction,
 because of our proof of the consistency of
 $\Phi(\omega_1, \omega_1, \omega)$
 in Theorem~\ref{main_thm},
 we ask the following:
\begin{question}
What is the consistency strength of
 $\zf$ + ``every real has a sharp'' +
 ``every subset of $\omega_1$
 is constructible from a real''?
\end{question}

By Theorem~\ref{equiv_of_measurable},
 the following closely related question
 is relevant:

\begin{question}
What is the consistency strength of
 $\zf + \dc$ + ``$V=L(\mathbb{R})$'' +
 ``the club filter on $\omega_1$ is an ultrafilter''?
\end{question}

Note that in the preprint \cite{mitchell_one_repeat}, 
 Mitchell shows
 that
 $\zf + \dc$ + the ``club filter on $\omega_1$ is an ultrafilter''
 is equiconsistent with
 $\zfc$ + ``there exists a weak repeat point''.

We also ask about $\Phi(\omega_1, \omega_1, \omega)$ itself,
 without any fragments of the Axiom of Choice:
\begin{question}
Does $\zf + \Phi(\omega_1, \omega_1, \omega)$
 imply there is an inner model with an inaccessible cardinal?
\end{question}

\subsection{Combinatorial consequences}

We are also interested in the combinatorial
 consequences of $\Phi(\omega_1, \omega_1, \omega)$
 together with fragments of the Axiom of Choice:

\begin{question}
Does $\zf$ + ``the club filter on $\omega_1$ is normal''
 + $\Phi(\omega_1, \omega_1, \omega)$ imply
 that every subset of $\omega_1$ is constructible from a real?
\end{question}

\subsection{Modifications of $\Phi$}

Given any inner model $M$,
 we can ask whether for every function
 $f : \omega_1 \times \omega_1 \to \omega$ in $M$
 there exists functions
 $g_1, g_2 : \omega_1 \to \omega$ in $V$
 such that
 $(\forall x,y \in \omega_1)\,
 f(x,y) \le \max \{ g_1(x), g_2(y) \}$.

\subsection{Beyond $\omega_1$}

The most intriguing question
 is whether $\omega_1$ can be enlarged:

\begin{question}
Is $\zf$ + $\Phi(\omega_2, \omega_2, \omega)$ consistent?
Does it follow from $\ad$?
\end{question}

A warm up question would be to ask whether
 $\Phi(\omega_2, \omega_1, \omega)$ is consistent.
And in general, we can ask about the status of
 $\Phi(\lambda_1, \lambda_2, \kappa)$
 for various $\lambda_1$, $\lambda_2$, and $\kappa$.
It would be interesting to know if there is any
 instance of $\Phi(\lambda_1, \lambda_2, \kappa)$,
 with $\lambda_1$ and $\lambda_2$ being ordinals,
 which fails in $L(\mathbb{R})$ assuming $\ad$.

\end{document}